\documentclass[11pt]{article}
\usepackage{amsmath,amsthm,amsfonts,amssymb,amscd, amsxtra}
\usepackage{url}
\usepackage[margin=1in]{geometry}
\usepackage{cite}
\usepackage{color}
\usepackage{graphicx}
\usepackage{caption}
\usepackage{subcaption}
\usepackage{multirow}
\usepackage{mathrsfs}%
\usepackage[title]{appendix}%
\usepackage{xcolor}%
\usepackage{textcomp}%
\usepackage{manyfoot}%
\usepackage{booktabs}%
\usepackage{algorithm}%
\usepackage{algorithmicx}%
\usepackage{algpseudocode}%
\usepackage{listings}%
\newtheorem{theorem}{Theorem}[section]
\newtheorem{lemma}{Lemma}[section]

\newtheorem{proposition}{Proposition}[section]

\def\argmin{\operatorname{argmin}}

\DeclareMathOperator{\Diag}{Diag}
\DeclareMathOperator{\diag}{diag}

\DeclareMathOperator{\Id}{Id}

\newcommand{\p}{\partial}

\newcommand{\K}{\mathcal K}

\newcommand{\sym}{\mathbb S}
\newcommand{\R}{\mathbb R}

\newcommand{\N}{\mathbb N}
\newcommand{\NN}{\mathbb{N}}

\begin{document}

\title{A semi-smooth Newton method for general projection equations applied to the nearest correlation matrix problem}

\author{Nicolas F. Armijo \thanks{Department of Applied Mathematics, University of São Paulo, Brazil (e-mail: {\tt
nfarmijo@ime.usp.br}).  The author was supported by Fapesp grant 2019/13096-2.}  
\and
Yunier  Bello-Cruz\thanks{Northern Illinois University, USA (e-mail: {\tt
yunierbello@niu.edu}).  The author was partially supported by the NSF Grant DMS-2307328 and by
an internal grant from NIU.}
\and
Gabriel Haeser \thanks{Department of Applied Mathematics, University of São Paulo, Brazil (e-mail: {\tt
ghaeser@ime.usp.br}). The author was supported by CNPq and Fapesp grant 2018/24293-0.}
}
\date{\today}

\maketitle

%\vspace{-0.4 in}

\begin{abstract}  

\noindent In this paper, we extend and investigate the properties of the semi-smooth Newton method when applied to a general projection equation in finite dimensional spaces. We first present results concerning Clarke's generalized Jacobian of the projection onto a closed and convex cone. We then describe the iterative process for the general cone case and establish two convergence theorems. We apply these results to the constrained quadratic conic programming problem, emphasizing its connection to the projection equation. To illustrate the performance of our method, we conduct numerical experiments focusing on semidefinite least squares, in particular the nearest correlation matrix problem. In the latter scenario, we benchmark our outcomes against previous literature, presenting performance profiles and tabulated results for clarity and comparison.
\medskip



\end{abstract}

%%%%%%%%%%%%%%%%%%%%%%%%%%%%%%%%%%%%%%%%%%%%%%%%%%%%%%%%%%%%%%
\section{Introduction}
%%%%%%%%%%%%%%%%%%%%%%%%%%%%%%%%%%%%%%%%%%%%%%%%%%%%%%%%%%%%%%

We begin by considering the following special nonlinear system:
\begin{equation}\label{eq:gen_closconv}
    P_{\K}(x) +Tx=b, 
\end{equation} 
where $\K \subseteq \mathbb{X}$ is a non-nempty, closed and convex cone of a finite dimensional vector space $\mathbb{X}$ with an inner product $\langle\cdot,\cdot\rangle$, $P_{\K}(x)$ is the projection of $x \in \mathbb{X}$ onto $\K$, $b\in \mathbb{X}$, and $T\colon \mathbb{X}\to \mathbb{X}$ is a linear operator.
Some particular cases of equation \eqref{eq:gen_closconv} have been studied, for instance, in \cite{Chen:2010,bello:2016,Bello-Cruz:2017,Griewank:2015, Sun:2015,Mangasarian:2009a,Ferreira:2015,Barrios:2015,Barrios:2016,armijo:2023}.
Among those problems, particular attention has been given to the cases where $\K$ is the $n$-dimensional non-negative orthant or Lorentz's cone. For these cases, novel iterative methods have been proposed; see, for instance, \cite{Barrios:2016, Bello-Cruz:2017,armijo:2023}.

Equation \eqref{eq:gen_closconv} is closely related to the quadratic cone-constrained programming: 
\begin{equation}
%\left(
\begin{matrix} \label{quad_conic_problem}
\text{min} & \frac{1}{2}\langle x,Qx \rangle + \langle q,x \rangle, \\
\text{s.t.} & x \in \K,
\end{matrix}
%\right)
\end{equation}
for  a linear operator $Q\colon \mathbb{X}\to \mathbb{X}$ and a vector $q\in \mathbb{X}$. The particularly relevant case occurs when $\mathbb{X}=\R^n$ and $\K$ is either the non-negative orthant or Lorentz's cone. The connection of \eqref{eq:gen_closconv} with \eqref{quad_conic_problem} is established by setting $T = (Q-\Id)^{-1}$ (where $\Id$ is the $n\times n$ identity matrix) and $b=-Tq$. Moreover, the projection onto $\K$ of a solution of equation \eqref{eq:gen_closconv} satisfies the first order necessary optimality conditions for problem \eqref{quad_conic_problem}. Here, we also prove that this property holds in the general case, that is, for any closed and convex cone $\K$ in a finite dimensional vector space $\mathbb{X}$ and considering any linear operator $Q\colon \mathbb{X}\to \mathbb{X}$. Additional linear equality constraints are also considered in \eqref{quad_conic_problem}.

We here focus our attention on the {\it semi-smooth Newton method} for solving equation \eqref{eq:gen_closconv}. This method finds a zero of the mapping $F\colon \mathbb{X}\to \mathbb{X}$,
\begin{equation} \label{eq:fucpw}
F(x)=P_{\K}(x) + Tx- b,  \qquad \qquad  ~x \in \mathbb{X}.
\end{equation}
By starting at a point  $x^{0}\in \mathbb{X}$,
the classical semi-smooth Newton method iterates as follows:
\begin{equation}\label{eq:newtonclassical}
x^{k+1} = x^k - \left[F'(x^k)\right]^{-1} F(x^k), 
\end{equation} where $F'(x^k)$ is a generalized Jacobian of $F$ at $x^k$. This iteration applying to \eqref{eq:fucpw} will take the following simple form:
\begin{equation}\label{eq:newtonc2conic}
\left(V(x^{k}) +T\right)x^{k+1}=b,  \qquad k\in\NN,
\end{equation} where $V(x^k) \in \p_C P_{\K}(x^k)$ is a Clarke's generalized Jacobian of the projection $P_{\K}(\cdot)$ at $x^k$. This is a consequence of the relation $V(x^k)x^k=P_{\K}(x^k)$ which we will prove later.

Our approach allows us to consider the relevant case when $\K=\sym^n_+$, the cone of positive semidefinite $n\times n$ matrices, where a subdifferential $V(x^k)$ can be computed by the spectral decomposition of $x^k$. For this case, we will tackle the quadratic cone-constrained problem \eqref{quad_conic_problem}, where we also include additional linear constraints; see problem \eqref{quad_conic_lin_problem} below. This is a well-known optimization problem with several applications. The goal is to study further the properties of the semi-smooth Newton method and prove the Q-linear global convergence of the iteration under some standard conditions on the linear operator $T$ for a general cone $\K$, allowing a more detailed study for the cone of positive semidefinite matrices.

The paper is organized as follows: we start by introducing our notation and presenting some preliminary results needed in our analysis. In Section \ref{sec:gen_proj}, we present general results about the projection operator onto cones and its differentiability.
In Section \ref{sec:newt_proj}, we present the semi-smooth Newton method and prove the previously mentioned global convergence theorems. Section \ref{sec:quad_app} is devoted to exploring the relationship between equation \eqref{eq:gen_closconv} and the quadratic conic programming problem \eqref{quad_conic_problem}, with special emphasis on the case when $\K = \sym^n_+$, including additional linear constraints. Finally, we present some numerical experiments for the positive semidefinite cone $\sym^n_+$, focusing particularly on semidefinite least squares problems. Specifically, we include a comparison with the algorithm from \cite{Qi:2006} in the context of the nearest correlation matrix problem, which is a well studied topic in economics \cite{Higham:2002}.

%%%%%%%%%%%%%%%%%%%%%%%%%%%%%%%%%%%%%%%%%%%%%%%%%%%%%%%%%
\subsection{Notations and preliminaries} \label{sec:int.1}
%%%%%%%%%%%%%%%%%%%%%%%%%%%%%%%%%%%%%%%%%%%%%%%%%%%%%%%%%%

In this section, we present some relevant results and definitions that are used in this paper. We denote the nonnegative integers by $\mathbb{N}$ and by $\Id$ the identity operator. The bracket notation $\langle \cdot, \cdot \rangle$ is referred to the inner product in any finite dimensional space $\mathbb{X}$.
 Given a linear operator $T\colon \mathbb{X} \rightarrow \mathbb{X}$, we use the notation $\|T\|$ for the operator norm of $T$, that is, $\|T\| := \text{max}\{\|Tx\| \mid x \in \mathbb{X}, \|x\|=1 \}$, where $\|x\|:=\sqrt{\langle x,x\rangle}$ is the norm associated with the inner product. We also denote by $T^*$ its adjoint linear operator $T^*\colon \mathbb{X} \rightarrow \mathbb{X}$, that is, $\langle Tx,y \rangle = \langle x, T^*y \rangle$, for all $x, y \in \mathbb{X}$. When $T$ is self-adjoint, that is, $T=T^*$, we say that $T$ is positive semidefinite (definite) if $\langle Tx, x \rangle \geq 0$ ($>0$, respectively), $\forall x \in \mathbb{X}, x\neq0$ and we denote by $\lambda_{\text{min}}(T)$ and $\lambda_{\text{max}}(T)$ the smallest and largest eigenvalues of $T$, respectively. For a cone $\K \subseteq \mathbb{X}$, the dual of $\K$ is denoted by $\K^*= \{y \in \mathbb{X}\mid \langle y,x \rangle \geq 0, \forall  x\in \K\}$. When $\mathbb{X}=\sym^n$ is the space of symmetric $n\times n$ matrices with the inner product $\langle A,B\rangle=\textrm{trace}(AB), A,B\in\sym^n$, we consider the self-dual cone $\K=\sym^n_+$ of positive semidefinite matrices.

The projection of a point $x$ onto a closed and convex cone $\K\neq\emptyset$ is denoted by $P_{\K}(x)$ and is defined by $P_{\K}(x)=\argmin\{\|y-x\|\mid y \in \K\}$. For a given mapping $F\colon \mathbb{X}\to \mathbb{X}$, we denote the set where it is differentiable by $D_F$ and the  Jacobian at a point $x\in D_F$ by $F'(x)$. The set of Clarke's generalized Jacobians at a point $x\in \mathbb{X}$ is denoted by $\partial_C F(x)$ and it is defined by
$$\partial_C F(x) = \text{conv}\left\{\lim_{k \to \infty} F'(x_k) \mid x_k \to x, x_k \in D_F\right\},$$ 
that is, the convex hull of all limits of Jacobians of $F$ nearby $x$. Throughout this paper, a Clarke's generalized Jacobian of the projection $P_{\K}(x)$ will be denoted by $V(x)$, so there is no confusion in not referring to a particular mapping $F$.  

We will make use of the well-known results below, which we state in the context of a general finite dimensional inner product space $\mathbb{X}$ as follows.

\begin{theorem}[Mean Value Theorem \cite{Clarke:1990}, Proposition 2.6.5, Page 79] \label{meanval_theo}
    Let $F\colon \mathbb{X} \rightarrow \mathbb{X}$ be a Lipschitz mapping. Then, we have
    $$F(y) - F(x) \in {\rm conv}\left(\p_C F([x,y])\right)(y-x),$$ that is, $F(y)-F(x) = U(z)(y-x)$ where $U(z) \in \p_C F(z)$ and $z$ is a convex combination of $x$ and $y$.
    \end{theorem}

\begin{lemma}[Banach's Lemma \cite{horn_johnson:2012}, Page 351] \label{banachlemma}
Let  $E\colon \mathbb{X} \rightarrow \mathbb{X}$ be a mapping onto $\mathbb{X}$. If $\|E\| <1$, then $E-\Id$ is invertible and
$$\|(E-\Id)^{-1}\| \leq \frac{1}{1-\|E\|}.$$
\end{lemma}

\begin{lemma}[Weyl's inequality \cite{horn_johnson:2012}, Theorem 4.3.1, Page 239]\label{weyl}
Let $A,B\colon \mathbb{X}\to \mathbb{X}$ be self-adjoint linear operators. Then, it holds
\begin{equation*}
\lambda_{\textnormal{min}}(A)+\lambda_{\textnormal{min}}(B)\leq\lambda_{\textnormal{min}}(A+B)\leq\lambda_{\textnormal{max}}(A+B)\leq\lambda_{\textnormal{max}}(A)+\lambda_{\textnormal{max}}(B).
\end{equation*}
\end{lemma}

Finally, an important result to ensure the existence and uniqueness of solutions of equation \eqref{eq:gen_closconv} is the contraction mapping principle. 

\begin{theorem}[Contraction mapping principle \cite{Ortega:1987}, Thm. 8.2.2,  page 153] \label{fixedpoint}
 Let $\Phi\colon \mathbb{X} \to   \mathbb{X}$ and suppose that there exists $\lambda \in  [0,1)$ such that  $\|\Phi(y)-\Phi(x)\| \le \lambda \|y-x\|$, for all $ x, y \in  \mathbb{X}$. Then, there exists a unique $\overline x\in  \mathbb{X}$  such that $\Phi(\overline x) = \overline x$.
\end{theorem}

%%%%%%%%%%%%%%%%%%%%%%%%%%%%%%%%%%%%%%%%%%%%%%%%%%%%%%%%%%%%%%%%%%%%%%%%%%%%%%%%%%%%%%%%%%%%%%
\section{On the projection mapping onto a closed and convex cone} \label{sec:gen_proj}
%%%%%%%%%%%%%%%%%%%%%%%%%%%%%%%%%%%%%%%%%%%%%%%%%%%%%%%%%%%%%%%%%%%%%%%%%%%%%%%%%%%%%%%%%%%%%

In this section, we study some useful results that will be important in the well-definiteness and global convergence of the semi-smooth Newton method for equation \eqref{eq:gen_closconv}. We begin by presenting the following result on the properties of generalized Jacobians.

\begin{theorem}\label{exis_diff_proj}
The projection operator $P_{\K}(\cdot)$ is differentiable almost everywhere. The Jacobian $P'_{\K}(x)$ (when it exists) and any generalized Jacobian $V(x)\in \p_C P_{\K}(x)$ for all $x\in \mathbb{X}$, are self-adjoint and positive semidefinite operators. Moreover, the following properties hold:
\item [ {\bf (i)}]\label{norm_diff1} $\|V(x)\| \leq 1$, $\forall V(x) \in \p_C P_{\K}(x)$ with $x \in \mathbb{X}$.
\item [ {\bf (ii)}]\label{Pxx=Px} $P_{\K}'(x)x=P_{\K}(x), \forall x \in D_{P_{\K}}.$
\item [ {\bf (iii)}]\label{Vxx=Px} $V(x)x=P_{\K}(x), \forall V(x) \in \p_C P_{\K}(x)$ with $x\in \mathbb{X}$.
\item [ {\bf (iv)}]\label{eigen_gendiff} For all $x\in \mathbb{X}$,
$$ 0 \leq \lambda_{\text{min}}(V(x)) \leq \lambda_{\text{max}}(V(x)) \leq 1, \forall V(x) \in \p_C P_{\K}(x).$$
\end{theorem}

\begin{proof}
The fact that the projection is differentiable almost everywhere is well-known due to its non-expansiveness (that is, the projection is $1$-Lipschitz). When $P'_{\K}(x)$ exists, it is self-adjoint and positive semidefinite due to Proposition 2.2 of \cite{fitzpatrick:1982}. Now, let $x \in \mathbb{X}$ and $V(x) \in \p_C P_{\K}(x)$. By definition we have that there exist $V_1(x),\dots,V_m(x)$ and $\{x^j_k\} \subset D_{P_{\K}}$ such that $\lim_{k\to\infty}x^j_k =x$,  $P_{\K}'(x^j_k) \rightarrow V_j(x)$, $\forall j=1,\dots,m,$ and $V(x) = \sum_{j=1}^m \alpha_jV_j(x)$, with $\sum_{j=1}^m \alpha_j =1$ and $\alpha_j \in [0,1]$ for all $j$. By the continuity of the inner product, we can deduce that each $V_j(x)$ is also self-adjoint and positive semidefinite, and therefore, by linearity, the same holds true for $V(x)$. 

To prove item (i), note that $\|P_{\K}'(x^j_k)\| \leq 1$ for all $j$ and $k$ due to non-expansiveness. Hence
\begin{align*}
\|V(x)\| &= \left\|\sum_{j=1}^m \alpha_j V_j(x) \right\| \leq \sum_{j=1}^m \alpha_j \|V_j(x)\| \\
&=\sum_{j=1}^m \alpha_j \lim_{k \to \infty}\|P_{\K}'(x^j_k)\|\leq1.
\end{align*}

For item (ii), note that since $\K$ is a cone, $P_{\K}(\cdot)$ is positive homogeneous, that is, $P_{\K}(tx) = tP_{\K}(x), \forall t \geq 0$ for any $x\in \mathbb{X}$. Let $x\in D_{P_{\K}}$. If $x=0$, the equality is evident. Assume that  $x \neq 0$. By the definition of $P^\prime_{\K}(x)$, we have that
\begin{align*}
0 &= \lim_{t \to 0}{\frac{\|P_{\K}(x+tx)-P_{\K}(x)-tP_{\K}'(x)x\|}{\|tx\|}}
= \frac{\|P_{\K}(x)-P_{\K}'(x)x\|}{\|x\|}.
\end{align*} Hence $P_{\K}'(x)x = P_{\K}(x)$, which proves item (ii).

In order to prove item (iii), by noting that $V(x)x-P_{\K}(x) = \sum_{j=1}^m \alpha_j (V_j(x)x-P_{\K}(x))$, it is enough to show that for all $j=1,\dots,m$, $V_j(x)x=P_{\K}(x)$. Recalling that $P'_{\K}(x^j_k)\to V_j(x)$, we have that
\begin{align*}
\|V_j(x)x-P_{\K}'(x^j_k)x^j_k\| &\leq \|V_j(x)x-V_j(x)x^j_k\|+\|V_j(x)x^j_k-P_{\K}'(x^j_k)x^j_k\|\\
&\leq \|V_j(x)\|\|x-x^j_k\|+\|V_j(x)-P_{\K}'(x^j_k)\|\|x^j_k\|\\
&\rightarrow 0.
\end{align*} Using item (ii) and the continuity of $P_{\K}(\cdot)$ we conclude that
\begin{align*}
\|V_j(x)x-P_{\K}(x)\| &\leq \|V_j(x)x-P_{\K}(x^j_k)\| + \|P_{\K}(x^j_k)-P_{\K}(x)\|\\
&= \|V_j(x)x-P_{\K}'(x^j_k)x^j_k\| + \|P_{\K}(x^j_k)-P_{\K}(x)\| \rightarrow 0.
\end{align*}
Finally, for item (iv), it is enough to note that $0\leq\lambda_{\text{min}}(V_j(x))\leq\lambda_{\text{max}}(V_j(x))\leq1$ for all $j=1,\dots,m$ due to the fact that $V_j(x)$ is self-adjoint and positive semidefinite with $\|V_j(x)\|\leq1$. The result now follows easily from Lemma \ref{weyl}.
\end{proof}

We conclude the section with the following useful result.

\begin{lemma}\label{norm_diff2}
Let $x, y \in \mathbb{X}$ and $V(x) \in\p_C P_{\K}(x)$. Then $\|P_{\K}(y)-P_{\K}(x)-V(x)(y-x)\| \leq \|y-x\|$.
\end{lemma}

\begin{proof}
By Theorem \ref{meanval_theo} we have that
$$P_{\K}(y)-P_{\K}(x)-V(x)(y-x) = (V(z)-V(x))(y-x),$$ with $V(z) \in \p_C P_{\K}(z)$, where $z$ is a convex combination of $x$ and $y$. The result follows from the fact that $\|V(z)-V(x)\|\leq1$ due to Lemma \ref{weyl} and Theorem \ref{exis_diff_proj} item (iv).
\end{proof}

Item~(ii) from Theorem~\ref{exis_diff_proj} provides a foundation for introducing the semi-smooth Newton method for solving equation \eqref{eq:gen_closconv}. Specifically, since the projection can be expressed as 
\[ P_{\mathcal{K}}(x) = V(x)x, \]
where \( V(x) \in \partial_C P_{\mathcal{K}}(x) \), we have that $F(x)$ as defined in \eqref{eq:fucpw} can be written as $$F(x)=(V(x)+T)x-b,$$ with $V(x)+T\in \p_C F(x)$. Then, iteration \eqref{eq:newtonclassical} can be expressed as $$x^{k+1}=x^k-(V(x^k)+T)^{-1}[(V(x^k)+T)x^k-b]=(V(x^k)+T)^{-1}b.$$ In the next section, we explore the convergence properties of this iteration.

%%%%%%%%%%%%%%%%%%%%%%%%%%%%%%%%%%%%%%%%%%%%%%%%%%%%%%%%%%%%%%%%%%%%%%%%%%%%%%%%%%%%%%%%%%%%%%%%%%%%%%%%%
\section{A semi-smooth Newton method for general projection equations} \label{sec:newt_proj}
%%%%%%%%%%%%%%%%%%%%%%%%%%%%%%%%%%%%%%%%%%%%%%%%%%%%%%%%%%%%%%%%%%%%%%%%%%%%%%%%%%%%%%%%%%%%%%%%%%%%%%%%%

In this section, we define a semi-smooth Newton method for solving equation \eqref{eq:gen_closconv} and study the convergence along with the sufficient conditions required to achieve it. Our goal is to extend the application of the semi-smooth Newton method, previously studied in \cite{Barrios:2016, Bello-Cruz:2017}, for the cases where $\K\subseteq\R^n$ is either the non-negative orthant or Lorentz's cone. This extension considers any closed and convex cone $\K\subseteq \mathbb{X}$. First, we establish a sufficient condition to the existence and uniqueness of the solution to the equation \eqref{eq:gen_closconv}.

\begin{theorem}[Sufficient condition for existence and uniqueness of a solution] \label{exi_uni_sol}
If $T$ is invertible and $\|T^{-1}\|<1$, then equation \eqref{eq:gen_closconv} has a unique solution for any $b \in \mathbb{X}$.
\end{theorem}

\begin{proof}
Equation \eqref{eq:gen_closconv} has a unique solution if and only if the mapping $\Phi(x) =-T^{-1}P_{\K}(x) +T^{-1}b$ has a unique fixed point. Hence, it is sufficient to prove that $\Phi$ is a contraction and use Theorem \ref{fixedpoint} to guarantee the existence and uniqueness of a fixed point. From the definition of $\Phi$, we have
$$\Phi(x) - \Phi(y) = -T^{-1}(P_{\K}(x) - P_{\K}(y)).$$ Since $\|P_{\K}(x) - P_{\K}(y)\| \leq \|x-y\|$ we deduce that $\|\Phi(x) - \Phi(y)\| \leq \|T^{-1}\| \|x-y\|$ concluding that $\Phi$ is a contraction since $\|T^{-1}\| <1$. 
\end{proof}

We define the semi-smooth Newton method for the mapping $F(x)=P_{\K}(x)+ Tx - b$ starting on an initial point $x^0\in \mathbb{X}$ as the iteration

\begin{equation} \label{eq:semismnewt}
(V(x^k) + T)x^{k+1}=b, \; k\in\N
\end{equation} with $V(x) \in \p_C P_{\K}(x)$.

Notice that if $x^k \rightarrow \overline{x}$, then $\overline{x}$ is a solution of equation \eqref{eq:gen_closconv}. To see this,  we subtract $(V(x^k)+T)x^k$ from both sides of \eqref{eq:semismnewt}. Using that $P_{\K}(x^k)=V(x^k)x^k$, we arrive at $(V(x^k) + T)(x^{k+1}-x^k)=b-P_{\K}(x^k)-Tx^k$. Since $V(x^k)+T$ is bounded, the left-hand side converges to zero, while from the continuity of the projection the right-hand side converges to $b-P_{\K}(\overline{x})-T\overline{x}$. Therefore, $\overline{x}$ is a solution of \eqref{eq:gen_closconv}.

We start by showing a sufficient condition for stopping the method \eqref{eq:semismnewt} at a solution.

\begin{proposition}[Stopping criterion]
If $V(x^{k+1}) = V(x^k)$, then $x^{k+1}$ is a solution of equation \eqref{eq:gen_closconv}.
\end{proposition}

\begin{proof}
From Theorem \ref{exis_diff_proj} item (iii) and \eqref{eq:semismnewt}, we have that
\begin{equation*}
P_{\K}(x^{k+1}) + Tx^{k+1} = (V(x^{k+1}) + T)x^{k+1}
= (V(x^k) + T)x^{k+1} 
= b.
\end{equation*} %Therefore $x^{k+1}$ is a solution.
\end{proof}
Now, we show a sufficient condition for the global convergence of iteration \eqref{eq:semismnewt}. Provided certain conditions regarding the norm of the inverse of $T$ are met, we can guarantee the existence and uniqueness of the solution of equation \eqref{eq:gen_closconv}. In addition to that, by imposing an additional norm condition, we obtain linear global convergence.

\begin{theorem} [Sufficient condition for global Q-linear convergence]\label{linconvergnewt}
Let $b \in \mathbb{X} $ and $T\colon \mathbb{X} \rightarrow \mathbb{X}$ be an invertible linear operator. Assume that $\|T^{-1}\|<1$. Then, equation \eqref{eq:gen_closconv} has a unique solution $\overline{x}$ and for any initial point $x^0$, the semi-smooth Newton sequence generated by equation \eqref{eq:semismnewt} is well-defined. Additionally, if $\|T^{-1}\|< \frac{1}{2}$ then the sequence $\{x^k\}$converges $Q$-linearly to $\overline{x}$ and satisfies
$$ \|x^{k+1}-\overline{x}\| \leq \frac{\|T^{-1}\|}{1-\|T^{-1}\|}\|x^{k}-\overline{x}\|, \; k\in\N.$$
\end{theorem}

\begin{proof}
First we know from Theorem \ref{exis_diff_proj} that $\|V(x)\| \leq 1$ for any $x \in \mathbb{X}$. Since $\|T^{-1}\| < 1$ we deduce that $\|T^{-1}V(x)\|<1$ for every $x \in \mathbb{X}$. Lemma \ref{banachlemma} implies that $-T^{-1}V(x)-\Id$ is invertible and therefore $V(x)+T$ is also invertible. % since
%$$ V(x)+T=-T(-T^{-1}V(x)-\Id), \forall x \in X. $$
In particular the semi-smooth Newton method \eqref{eq:semismnewt} is well defined. Let $\overline{x}$ be the only solution of problem \eqref{eq:gen_closconv} (which exists and is unique due to Theorem \ref{exi_uni_sol}). So, this point satisfies the relation $(V(\overline{x})+T)\overline{x}-b = 0$. Combining with \eqref{eq:semismnewt} we deduce that
$$(V(x^k)+T)(x^{k+1}-\overline{x})=(V(\overline{x})-V(x^k))\overline{x}=V(\overline{x})\overline{x}-V(x^k)x^k-V(x^k)(\overline{x}-x^k).$$
Since $V(\overline{x})\overline{x}=P_{\K}(\overline{x})$ and $V(x^k)x^k=P_{\K}(x^k)$, using Lemma \ref{norm_diff2}, we obtain
$$ \| x^{k+1}  - \overline{x}\| \leq \|(V(x^k) + T)^{-1}\| \| (P_{\K}(\overline{x}) - P_{\K}(x^k) - V(x^k)(\overline{x} - x^k))\|\leq\|(V(x^k) + T)^{-1}\| \| \overline{x}-x^k\|.$$
%Using the same algebraic manipulation of \cite{BelloCruz:2017} we have
But $ \|(V(x^k) + T)^{-1} \| = \| (T(T^{-1}V(x^k)+\Id))^{-1}\| \leq \| (T^{-1}V(x^k) - \Id)^{-1} \|\| T^{-1}\|.$ Lemma \ref{banachlemma} and $\| T^{-1}V(x^k) \| < 1$ implies that
$$ \| (T^{-1}V(x^k) - \Id)^{-1} \| \leq \frac{1}{1-\|T^{-1}V(x^k)\|} \leq \frac{1}{1-\|T^{-1}\|}.$$ %Therefore, combining the previous expression, the fact that $\|T^{-1}V(x^k)\| \leq \|T^{-1}\|$ and the hypothesis that $\|T^{-1}\| < \frac{1}{2}$ we conclude that
Thus, we have that $ \| x^{k+1}  - \overline{x}\| \leq \frac{\|T^{-1}\|}{1-\|T^{-1}\|} \|x^k-\overline{x}\|$ with $\frac{\|T^{-1}\|}{1-\|T^{-1}\|} <1$ due to the assumption that $\|T\|<\frac{1}{2}$. Hence, $x^k$ converges Q-linearly to the unique solution $\overline{x}$.
\end{proof}

The previous result states that with only a norm condition on the operator $T^{-1}$, namely $\|T^{-1}\| < \frac{1}{2}$, we can achieve Q-linear convergence of the method. We prove next that for the case where $T$ being a positive definite linear mapping, the weaker norm condition $\|T^{-1}\| <1$ is sufficient to ensure Q-linear convergence of the iteration \eqref{eq:semismnewt} to the unique solution of the problem \eqref{eq:gen_closconv}.

\begin{theorem} \label{linconvergnewt_posdef}
Let $b \in \mathbb{X} $ and $T\colon \mathbb{X} \rightarrow \mathbb{X}$ be a positive definite linear operator. Then, equation \eqref{eq:gen_closconv} has a unique solution $\overline{x}\in \mathbb{X}$ and for any initial point $x^0\in \mathbb{X}$, the semi-smooth Newton sequence generated by \eqref{eq:semismnewt} is well-defined. Additionally, if $\|T^{-1}\|< 1$ then the sequence $\{x^k\}$converges $Q$-linearly to $\overline{x}$ and satisfies
$$ \|x^{k+1}-\overline{x} \| \leq \|T^{-1}\|\|x^{k}-\overline{x}\|, \; k\in\N.$$
\end{theorem}

\begin{proof}
First notice that from Lemma \ref{weyl} and the positive definiteness of $T$, it follows that $\Id + T$ is invertible with $\|(\Id + T)^{-1}\|<1$. Using Moreau's decomposition \cite[Theorem 3.2.5]{Hiriart-Urruty:1993}, we can write any $x\in \mathbb{X}$ as $x = P_{\K}(x) - P_{\K^*}(-x)$. Now, it follows directly that $x$ is a solution of equation \eqref{eq:gen_closconv} if and only if $x$ is a fixed point of $\Phi(x) = (\Id+T)^{-1}(b-P_{\K^*}(-x))$.
Since
$$ \Phi(x)-\Phi(y) = (\Id + T)^{-1}(-P_{\K^*}(-x)+P_{\K^*}(-y)), \quad x,y \in \mathbb{X},$$ we deduce that $\Phi$ is a contraction due to the non-expansiveness of the projection. This gives existence and uniqueness of a solution of problem \eqref{eq:gen_closconv}.

By Theorem \ref{exis_diff_proj} item (iv) and Lemma \ref{weyl}, it follows that $V(x)+T$ is positive definite with $\|(V(x)+T)^{-1}\|\leq\|T^{-1}\|$ for all $x\in \mathbb{X}$. Hence, iteration \eqref{eq:semismnewt} is well-defined for every starting point $x^0\in \mathbb{X}$. The Q-linear convergence when $\|T^{-1}\|<1$ now follows from the relation $\|x^{k+1}-\overline{x}\|\leq\|(V(x^k)+T)^{-1}\|\|x^k-\overline{x}\|$ deduced in the proof of Theorem \ref{linconvergnewt}.
\end{proof}

Note that although the assumption of positive definiteness is sufficient for existence and uniqueness of the solution without imposing a condition on the norm of $T^{-1}$, it does not guarantee the convergence of Newton's method; see \cite[Example 1]{armijo:2023}.

In the following section, we show relevant applications of our results to quadratic conic programming.

%%%%%%%%%%%%%%%%%%%%%%%%%%%%%%%%%%%%%%%%%%%%%%%%%%%%%%%%%%%%%%%%%%%%%%%%%%%%%%%%%%%%%%%%%%%%%%%%%%%%%%%%%
\section{Application to quadratic conic programming} \label{sec:quad_app}
%%%%%%%%%%%%%%%%%%%%%%%%%%%%%%%%%%%%%%%%%%%%%%%%%%%%%%%%%%%%%%%%%%%%%%%%%%%%%%%%%%%%%%%%%%%%%%%%%%%%%%%%%

In this section, we connect equation \eqref{eq:gen_closconv} with the important quadratic conic programming problem

\begin{equation} \label{quad_conic_prog}
%\left(
\begin{matrix}
\text{min} & \frac{1}{2}\langle x,Qx \rangle + \langle q,x \rangle, \\
\text{s.t.} & x \in \K.
\end{matrix}
%\right)
\end{equation}

This problem has been widely studied and has multiple applications such as semidefinite least squares and, in particular, the nearest correlation matrix problem which we will present next.

The Lagrangian of the problem is given by
$$ L(x,\mu) := \frac{1}{2}\langle x,Qx \rangle + \langle q,x \rangle - \langle\mu,x\rangle,$$ where $\mu \in \K^*$. Then, the well-known complementary KKT conditions are given by
\begin{align*}
Q\overline{x}+q-\overline{\mu} &= 0,\\
\langle\overline{\mu},\overline{x}\rangle &= 0.
\end{align*} Or equivalently,
\begin{equation} \label{eq:kkt}
\langle Q\overline{x}+q,\overline{x} \rangle=0 \text{, }\text{ }\text{ }\text{ }Q\overline{x}+q \in \K^*\text{, }\text{ }\text{ }\text{ }\overline{x} \in \K.
\end{equation}

In order to find a solution to the KKT system \eqref{eq:kkt}, we use the following modified projection equation: 
\begin{equation} \label{eq:gen_quadprog}
(Q-\Id)P_{\K}(x)+x=-q.
\end{equation}

With this in mind, we have the following connection between the solutions of  \eqref{eq:gen_quadprog} and the ones of the KKT conditions \eqref{eq:kkt} above. The following theorem is a generalization of Proposition 4 in \cite{Bello-Cruz:2017}.

\begin{theorem} [KKT points and solutions of a generalized projection equation] \label{solKKT_iff_solQuadEq} 
If $x$ is solution of equation \eqref{eq:gen_quadprog}, then $\overline{x}=P_{\K}(x)$ is a solution of the KKT system \eqref{eq:kkt}. On the other hand, if $\overline{x}$ is a solution of system \eqref{eq:kkt}, then $x = \overline{x} - (Q\overline{x}+q)$ is a solution of \eqref{eq:gen_quadprog}.
\end{theorem}

\begin{proof}
For the first part, let $x$ be a solution of \eqref{eq:gen_quadprog}. Using Moreau's decomposition \cite[Theorem 3.2.5]{Hiriart-Urruty:1993} for $x$, we have
\begin{equation} \label{Moreau1}
x = P_{\K}(x) - P_{\K^*}(-x),
\end{equation} and
\begin{equation} \label{Moreau2}
\langle P_{\K}(x) , P_{\K^*}(-x) \rangle = 0.
\end{equation}
By hypothesis we get that
$$ QP_{\K}(x)+q = P_{\K}(x)-x = P_{\K^*}(-x) \in \K^*,$$ where we used \eqref{Moreau1} in the last equality. Now using the previous equation and \eqref{Moreau2}, we have

$$ \langle QP_{\K}(x)+q,P_{\K}(x) \rangle = \langle P_{\K^*}(-x), P_{\K}(x) \rangle = 0,$$ implying that $P_{\K}(x)\in\K$ solves \eqref{eq:kkt}.

For the second part, let $x$ be a solution of \eqref{eq:kkt} and $\overline{x} := x -(Qx+q)$.
Since $x\in\K$, $Qx+q\in\K^*$ with $\langle Qx+q,x \rangle=0$, it follows by Moreau's decomposition that $x = P_{\K}(\overline{x})$.
Thus, replacing $\overline{x}$ in \eqref{eq:gen_quadprog} we obtain
\begin{equation*}
(Q-\Id)P_{\K}(\overline{x})+\overline{x} = (Q-\Id)x+x-(Qx+q) = -q.
\end{equation*} Therefore, $\overline{x}$ is a solution of \eqref{eq:gen_quadprog}, which concludes the proof.
\end{proof}

Theorem \ref{solKKT_iff_solQuadEq} establishes a significant connection between system \eqref{eq:kkt} and equation \eqref{eq:gen_quadprog}. In particular, it asserts that if \eqref{eq:gen_quadprog} does not have a solution, then the quadratic conic programming problem \eqref{quad_conic_prog} lacks points satisfying the complementary optimality conditions.

We now extend our previous results to the case of equation \eqref{eq:gen_quadprog}. The proofs of the next three results use Theorem \ref{exi_uni_sol} and closely follow the ideas presented in \cite{Bello-Cruz:2017}, and therefore, we omit some details for brevity. We begin by presenting two propositions regarding the existence and uniqueness of solutions as follows:

\begin{proposition} \label{Q-Id<1}
If $\|Q-\Id\|<1$, then equation \eqref{eq:gen_quadprog} has a unique solution for any $q \in \mathbb{X}$.
\end{proposition}

\begin{proof}
Similar to the proof of Theorem \ref{exi_uni_sol} but replacing $T$ by $(Q-\Id)^{-1}$.
\end{proof}

\begin{proposition}
If $Q$ is invertible and $\|Q^{-1}-\Id\|<1$, then equation \eqref{eq:gen_quadprog} has a unique solution for any $q \in \mathbb{X}$.
\end{proposition}

\begin{proof}
Similar to the proof of Theorem \ref{exi_uni_sol} but defining
$ \Phi(x)=(Q^{-1}-\Id)P_{\K^*}(-x)-Q^{-1}q $ and using that $x = P_{\K}(x)-P_{\K^*}(-x)$ and the fact that the projection $P_{\K^*}(\cdot)$ is non-expansive.
\end{proof}

Next, we specialize our Q-linear convergence results to the case of equation \eqref{eq:gen_quadprog}.

\begin{theorem} \label{linconvergnewt_quadprog}
Let $q \in \mathbb{X} $ and $Q\colon \mathbb{X} \rightarrow \mathbb{X}$ a linear operator. Assume that $Q-\Id$ is invertible and $\|Q-\Id\|<1$. Then, \eqref{eq:gen_quadprog} has a unique solution $\overline{x}$, and for any initial point $x^0$ the semi-smooth Newton sequence generated by \eqref{eq:semismnewt} is well-defined. Additionally, if $\|Q-\Id\|< \frac{1}{2}$ then the sequence $\{x^k\}$ converges $Q$-linearly to $\overline{x}$ and satisfies
$$ \|x^{k+1}-\overline{x}\| \leq \frac{\|Q-\Id\|}{1-\|Q-\Id\|}\|x^{k}-\overline{x}\|, \; k\in\N.$$
\end{theorem}

\begin{proof}
Similar to the proof of Theorem \ref{linconvergnewt}.
\end{proof}

Before stating a result analogous to Theorem \ref{linconvergnewt_posdef}, we need the following lemma. 

\begin{lemma}\label{lemma_posdef}
Let $x \in \mathbb{X}$ and $V(x) \in \p_C P_{\K}(x)$. If $Q$ is a positive definite linear mapping, then
$ (Q-\Id)V(x)+\Id $ is invertible.
\end{lemma}

\begin{proof}
By contradiction let us suppose that there exists $u \neq 0$ with $((Q-\Id)V(x)+\Id)u=0$, or equivalently that
\begin{equation}\label{eq:lemma_quadinv}
QV(x)u = (V(x)-\Id)u.
\end{equation} Since $Q$ is positive definite and $V(x)$ is self-adjoint, we have that
$$ 0 \leq \langle QV(x)u , V(x)u \rangle = \langle V(x)^*QV(x)u , u \rangle = \langle V(x)QV(x)u , u \rangle = \langle (V(x)^2-V(x))u , u \rangle \leq 0,$$ where the last inequality is due to Theorem \ref{exis_diff_proj}, item (iv). Then, $\langle QV(x)u , V(x)u \rangle=0$, which implies that $V(x)u=0$. But by \eqref{eq:lemma_quadinv}, we conclude that $u=0$, which is a contradiction. Hence, the operator $(Q-\Id)V(x)+\Id$ is invertible.
\end{proof}

Finally, we present a sufficient condition for the convergence of the semi-smooth Newton method which is analogous to Theorem \ref{linconvergnewt_posdef} applied to the quadratic conic programming problem when $Q$ is positive definite.

\begin{theorem} \label{linconvergnewt_quadprog_posdef}
Let $q \in \mathbb{X} $ and $Q\colon \mathbb{X} \rightarrow \mathbb{X}$ a positive definite linear operator. Then, for any initial point $x^0\in \mathbb{X}$, the semi-smooth Newton sequence generated by \eqref{eq:semismnewt} is well-defined. Additionally, if $Q-\Id$ is invertible, then equation \eqref{eq:gen_quadprog} has a unique solution $\overline{x} \in \mathbb{X}$ and, if $\|Q-\Id\|<1$ then $\{x^k\}$ converges $Q$-linearly to $\overline{x}$ satisfying
$$ \|x^{k+1}-\overline{x}\| \leq \|Q-\Id\|\|x^{k}-\overline{x}\|, \quad k\in\N.$$
\end{theorem}

\begin{proof}
Using Lemma \ref{lemma_posdef}, the proof follows the lines of the proof of Theorem \ref{linconvergnewt_posdef}.
\end{proof}

We next consider an extension of the quadratic conic programming problem \eqref{quad_conic_prog} by including additional linear constraints. That is, given an additional linear mapping $\mathcal{A}\colon \mathbb{X}\to \mathbb{Y}$, where $\mathbb{Y}$ is also a finite dimensional inner product vector space, and given $b\in \mathbb{Y}$, we consider the problem

\begin{equation}
\label{quad_conic_lin_problem}
%\left(
\begin{matrix} 
\text{min} & \frac{1}{2}\langle x,Qx \rangle + \langle q,x \rangle, \\
\text{s.t.} & \mathcal{A}x=b, \\
 & x \in \K.
\end{matrix}
%\right)
\end{equation} The Lagrangian function associated with \eqref{quad_conic_lin_problem} is given by:

\begin{equation} \label{eq:lagrangian_quad_conic_lin}
L(x,\lambda,\mu) = \frac{1}{2}\langle x,Qx \rangle + \langle q,x \rangle + \langle \lambda,\mathcal{A}x-b \rangle-\langle\mu,x\rangle.
\end{equation} where $\lambda \in \mathbb{Y}$ and $\mu\in \mathbb{X}$. The complementary KKT conditions for problem \eqref{quad_conic_lin_problem} are given by:
\begin{align*}
Q\overline{x}+q+\mathcal{A}^*\overline{\lambda}-\overline{\mu} &= 0,\\
\mathcal{A}\overline{x} &=b,\\
\langle\overline{\mu},\overline{x}\rangle &= 0,
\end{align*} where the Lagrange multipliers are $\overline\mu\in \K^*$ and $\overline\lambda\in \mathbb{Y}$. This system can be rewritten as the following complementary system:
\begin{equation} \label{eq:_linear_kkt}
\Biggl \langle
\Biggl(
\begin{matrix}
Q\overline{x} + \mathcal{A}^*\overline{\lambda} + q \\
\mathcal{A}\overline{x} - b
\end{matrix}
\Biggr),
\Biggl(
\begin{matrix}
\overline{x} \\
\overline{\lambda}
\end{matrix}
\Biggr)
\Biggr \rangle = 0 \text{, }\text{ }\text{ }\text{ }
\Biggl(
\begin{matrix}
Q\overline{x} + \mathcal{A}^*\overline{\lambda} + q \\
A\overline{x} - b
\end{matrix}
\Biggr) \in K^*\text{, }\text{ }\text{ }\text{ }(\overline{x},\overline{\lambda}) \in K,
\end{equation}
where $K=\K \times \mathbb{Y}$ and its dual is given by $K^*=\K^*\times\{0\}$. Thus, by Theorem \ref{solKKT_iff_solQuadEq}, this system may be solved by means of the following projection equation
\begin{equation}\label{aux1}
\left(\left(\begin{array}{cc}Q&\mathcal{A}^*\\\mathcal{A}&0\end{array}\right)-\Id\right)P_{K}(x,\lambda)+\left(\begin{array}{c}x\\ \lambda\end{array}\right)=\left(\begin{array}{c}-q\\b\end{array}\right),
\end{equation}
where a solution $(\overline{x},\overline{\lambda})$ of \eqref{aux1} is such that $P_{K}(\overline{x},\overline{\lambda})=(P_{\K}(\overline{x}),\overline{\lambda})$ solves the complementary system \eqref{eq:_linear_kkt}.
Notice that equation \eqref{aux1} can be rewritten as

\begin{equation} \label{eq:gen_quadprog_linear}
\Biggl(
\begin{matrix}
(Q-\Id)P_{\K}(x) + \mathcal{A}^*\overline{\lambda} + x \\
\mathcal{A} P_{\K}(x)
\end{matrix}
\Biggr) =
\Biggl(
\begin{matrix}
-q \\
b
\end{matrix}
\Biggr),
\end{equation} 
and for any starting point $x^0\in \mathbb{X}$, the correspondent semi-smooth Newton iteration can be rewritten as follows: 
\begin{equation} \label{eq:newt_gen_quadprog_linear}
\Biggl(
\begin{matrix}
(Q-\Id)V(x^k)x^{k+1}+x^{k+1} + \mathcal{A}^*\lambda^{k+1} \\
\mathcal{A}V(x^k)x^{k+1}
\end{matrix}
\Biggr) =
\Biggl(
\begin{matrix}
-q \\
b
\end{matrix}
\Biggr),
\end{equation}
for $k \in \N$. The convergence of the sequence $\{(x^k,\lambda^k)\}$ generated by \eqref{eq:newt_gen_quadprog_linear} is guaranteed by applying Theorems \ref{linconvergnewt_quadprog} and \ref{linconvergnewt_quadprog_posdef} with respect to equation \eqref{aux1}.  

%%%%%%%%%%%%%%%%%%%%%%%%%%%%%%%%%%%%%%%%%%%%%
\section{The nearest correlation matrix problem}
%%%%%%%%%%%%%%%%%%%%%%%%%%%%%%%%%%%%%%%%%%%%%
\label{sec:computationalresults}

In this section, we describe the application of the semi-smooth Newton method to the nearest correlation matrix problem. This specific problem represents a special case of the broader positive semidefinite least squares problem (referenced as \eqref{SDP_lst_sqrs} below), a well-studied area notable for its significant applications and established algorithms; see \cite{Higham:2002, Malick:2004}.

Let $\mathbb{X}=\mathbb{S}^n$ be the set of symmetric $n\times n$ matrices with real entries and $\K=\mathbb{S}^n_+\subset\mathbb{S}^n$ be the cone of positive semidefinite matrices. Given any finite dimensional inner product space $\mathbb{Y}$ and a linear mapping $\mathcal{A}\colon\mathbb{S}^n\to\mathbb{Y}$, we consider the problem

\begin{equation}\label{SDP_lst_sqrs}
%\left(
\begin{matrix}
\text{min} & \frac{1}{2}\|X-G \|^2, \\
\text{s.t.} & \mathcal{A}(X) = b,\\
& X \in \sym^n_+,\\
\end{matrix}
%\right)
\end{equation} where $G\in\mathbb{S}^n$ and $b\in \mathbb{Y}$ are given and the Frobenius norm is defined as $\| A \| =\sqrt{\langle A,A\rangle}$, where $\langle A,B\rangle={\rm trace}(AB)$ for $A,B\in\mathbb{S}^n$. 

Since $\frac{1}{2}\|X-G\|=\frac{1}{2}\langle X,X \rangle - \langle X,G \rangle$, problem \eqref{SDP_lst_sqrs} is an instance of \eqref{quad_conic_lin_problem} with $Q=\Id$ and $q=-G$. Thus, substituting in \eqref{eq:newt_gen_quadprog_linear} we arrive at the following iteration for the semi-smooth Newton method for finding a solution $(X,\Lambda)\in\mathbb{S}^n\times\mathbb{Y}$ of \eqref{aux1}:
\begin{equation}
\label{aux2}
\Biggl(
\begin{matrix}
X^{k+1} + \mathcal{A}^*(\Lambda^{k+1}) \\
\mathcal{A}V(X^k)X^{k+1}
\end{matrix}
\Biggr) =
\Biggl(
\begin{matrix}
G \\
b
\end{matrix}
\Biggr).
\end{equation}

The nearest correlation matrix problem is the particular case where $\mathbb{Y}=\R^n$ and $\mathcal{A}=\diag\colon\sym^n\to\R^n$ which maps a symmetric matrix to its diagonal vector. Its dual $\mathcal{A}^*=\Diag\colon\R^n\to\sym^n$ maps a vector to a diagonal matrix with the given vector in its diagonal. The right-hand side vector $b\in\R^n$ will be fixed at the vector $e$ of all ones. Namely, let us consider problem

\begin{equation} \label{Nearest_corr_prblm}
\begin{matrix}
\text{min} &\frac{1}{2} \|X-G \|^2, \\
\text{s.t.} & \diag(X) = e,\\
 & X \in \sym^n_+ .\\
\end{matrix}
\end{equation}

Thus, substituting in \eqref{eq:gen_quadprog_linear} we aim at solving the following projection equation for $(\overline{X},\overline{\lambda})\in\mathbb{S}^n\times\mathbb{R}^n$:
\begin{align} \label{corr_mat_equation_sol}
\overline{X} +  \Diag(\overline{\lambda}) &= G, \\
\diag(P_{\sym^n_+}(\overline{X})) &= e,
\end{align}
by means of the semi-smooth Newton iteration
\begin{align} 
\label{a} X^{k+1} +  \Diag(\lambda^{k+1}) &= G, \\
\label{b} \diag(V(X^k)X^{k+1}) &= e,
\end{align}
which is obtained from \eqref{aux2}.

From \eqref{corr_mat_equation_sol} and \eqref{a} we deduce that the off-diagonal entries of $\overline{X}$ and $X^{k+1}$ are equal to the correspondent off-diagonal entries of $G$. Thus, by defining $D^{k+1}=\Diag(\diag(X^{k+1}))$ and $\hat{G} = G-\Diag(\diag(G))$ we obtain
\begin{align}
\label{corr_iter_a} X^{k+1} &= D^{k+1}+\hat{G}, \\
\label{corr_iter_b} \lambda^{k+1} &= \diag(G)- \diag(D^{k+1}). 
\end{align} Now, a simple calculation using $\eqref{corr_iter_a}$ gives
\begin{align*}
\diag(V(X^k)X^{k+1}) &= \diag(V(X^k)(D^{k+1}+\hat{G})), \\
 &= \diag(V(X^k)D^{k+1}) + \diag(V(X^k)\hat{G}), \\
 &= \Diag( \diag(V(X^k))\diag(D^{k+1}) + \diag(V(X^k)\hat{G}). \\
\end{align*} Thus, substituting in $\eqref{b}$ we arrive at the following expression for the iteration $D^{k+1}$:
\begin{equation}
\label{corr_iter_c} \diag(D^{k+1}) = ( \Diag( \diag(V(X^k)))^{-1}[e- \diag(V(X^k)\hat{G})],
\end{equation}
which gives a fully computable iteration for our Newton method applied to the nearest correlation matrix problem. In particular, we compute $V(X) = U D U^T$ from the spectral decomposition $X = U\Lambda U^T$, where $D$ is a diagonal matrix with $D_{ii}=1$ if $\Lambda_{ii} > 0$ and $D_{ii}=0$ if $\Lambda_{ii} \leq 0$. When iteration \eqref{corr_iter_c} is not defined, we use in our implementation the pseudoinverse, however this never occurred in the tests we run. We prove next that the iteration is well defined when $\diag(X^k)>0$:

\begin{proposition}
Let $X \in \sym^n$. If $ \diag(X) > 0$, then $ \diag(V(X)) >0$.
\end{proposition}

\begin{proof}
Let $X=U\Lambda U^T$ and $i \in \{1,\dots,n\}$. Denote by $u_j$ the $j$-th row of $U, j=1,\dots,n$. We have
$$ X_{ii} = \sum_{j=1}^n \Lambda_{jj}(u_j)_i^2 > 0.$$ Then there exists $k$ such that $\Lambda_{kk}(u_k)_i^2 > 0$. In particular, $\Lambda_{kk}>0$ and $(u_k)_i^2>0$. Then, $V(X)=UDU^T$ with $D_{kk}=1$. Therefore,
\begin{align*}
V(X)_{ii} &= \sum_{j=1}^n D_{jj}(u_j)_i^2 = D_{kk}(u_k)_i^2 + \sum_{j \neq k} D_{jj}(u_j)_i^2 \geq (u_k)_i^2> 0,
\end{align*}
proving the result.
\end{proof}

It turns out that in \cite{Qi:2006}, a method closely resembling ours was proposed for the nearest correlation matrix problem \eqref{Nearest_corr_prblm_2}. In their approach, a semi-smooth Newton method is applied to the following function:
$$\tilde{F}(y):=\mathcal{A}P_{\mathbb{S}^n_+}(G+\mathcal{A}^*y)-e,$$
where $\mathcal{A}=\diag$.
A significant contribution of their work is the provision of a formula for computing the matrix-vector operation $h\mapsto V_yh$, where $V_y$ belongs to the (B-)subdifferential of $\tilde{F}$ at $y$, whithout the need to compute $V_y$ explicitly. The drawback to this approach is that they must then resort to an iterative procedure for computing each Newton iteration, without exploiting the diagonal structure of the operator $\mathcal{A}$. In our approach, we compute the subdifferential explicitly, which allows us to solve explicitly the diagonal Newtonian linear system. We note that both methods require computing the full spectral decomposition of an $n\times n$ matrix at each iteration. The algorithm in \cite{Qi:2006} also adds a line search procedure in order to ensure global convergence. For a fairer comparison in our study, we turned off this procedure, although its impact on the method's performance was minimal.

\section{Numerical experiments}

In order to observe the behavior of the method, we conducted experiments 5.5 through 5.8 as described in \cite{Qi:2006}. In these four experiments the {\it Nearest Correlation Problem} is solved using randomly generated data. We remind the definition of the problem

\begin{equation} \label{Nearest_corr_prblm_2}
\begin{matrix}
\text{min} & \frac{1}{2}\|X-G \|^2, \\
\text{s.t.} & \diag(X) = e,\\
 & X \in \sym^n_+ .
\end{matrix}
\end{equation}

Experiment 5.5 involves generating the matrix $G$ as $G:=C+\alpha R$, where $C$ is a random correlation matrix generated by the \texttt{randcorr} Matlab command, $R_{ij} \in [-1,1]$ is random, and $\alpha\geq0$ is chosen. In Experiment 5.6, $G$ is defined as a matrix of random numbers $G_{ij} \in [-1,1]$ with $G_{ii}$ fixed at $1$ for $i=1,\dots,n$. Experiment 5.7 is analogous to 5.6, but in this case $G_{ij} \in [0,2]$. Finally, Experiment 5.8 defines $G$ as
$$ G :=
\left(\begin{matrix}
\frac{\ell}{1-\ell}(E_\ell-\Id_\ell) & 0\\
0 & 0
\end{matrix}\right) + D + \alpha R,$$
where $1 \leq \ell \leq n$, $\alpha\geq0$, $\Id_\ell$ is the identity matrix of dimension $\ell$, $E_\ell$ is a matrix of $1$'s of dimension $\ell$, $D$ is a random diagonal matrix with $D_{ii} \in [-20000,20000]$, and $R$ is a random matrix such that $R_{ij} \in [-1,1]$. The comparison is made using performance profiles \cite{dolan:2002}. We built $10$ random problems for each choice of parameters. In Experiment 5.5, $n=3000$ is fixed and we test all values of $\alpha\in\{0.01, 0.1, 1, 10\}$. For Experiments 5.6 and 5.7 we test $n=1000, 2000, 3000$, and for Experiment 5.8 we test $n = 5000, 8000, 10000$ for $\alpha=0.001$ and $\ell = n/2$.

The experiments were ran in a 2.30 GHz Intel(R) Core(TM) i5-8300H processor, 16 GB of RAM and operating system Windows 10 using Matlab 9.5.0.944444 (R2018b).
We stopped the execution of the methods when the Euclidean residual error is smaller than $10^{-5}$ and we report the CPU time required by both methods.

In Experiment 5.5, our method was in general slower than that of Qi and Sun \cite{Qi:2006}. For $\alpha=0.01$, our method took on average 65\% more CPU time. This percentage was 137\% and 499\% for $\alpha=1$ and $10$, respectively. However, for $\alpha=0.1$, their method took on average 13\% more CPU time than ours. In Figure \ref{fig:3k_01} we present a performance profile for Experiment 5.5 with $\alpha = 0.1$.

\begin{figure}[!ht]
\centering
\includegraphics[scale=0.5]{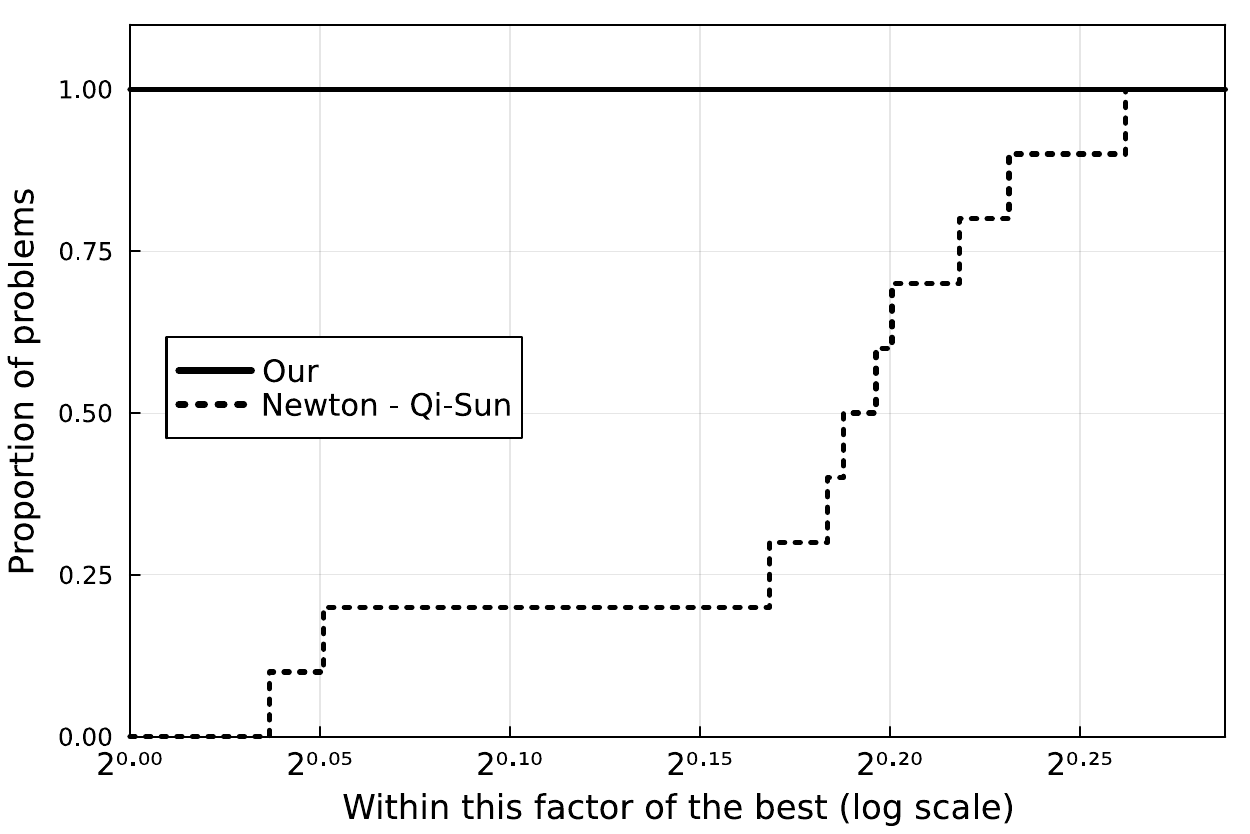}
\caption{Performance profile for Experiment 5.5 with $\alpha = 0.1$.}
\label{fig:3k_01}
\end{figure}

In Experiments 5.6 and 5.7, we observed that our method was slower compared to \cite{Qi:2006}. For Experiment 5.6 our method took on average $1.6$ times the CPU time required by the method from \cite{Qi:2006} for $n=1000$, $2$ times for $n=2000$ and $2.5$ times for $n=3000$. The situation was worse regarding Experiment 5.7 where we observe that our method took around $4$ times the total CPU time required by the method from \cite{Qi:2006} for $n=1000$. It was $7$ times slower for $n=2000$ and $9$ times slower for $n=3000$. The situation is much more favorable in Experiment 5.8. In Figures \ref{fig:5k_0001}, \ref{fig:8k_0001}, and \ref{fig:10k_0001}, we present the performance profiles concerning $n=5000$, $n=8000$, and $n=10000$, respectively. In Table \ref{fig:Tab_5.8} we show the average total time and the average number of iterations for both methods and all dimensions tested.

We observed that the method of Qi and Sun was faster in all problem with dimension $5000$ and in $~75\%$ of the problems with dimension $8000$, while our method was faster in all problems of dimension $10000$. In Table \ref{fig:Tab_5.8} we can see that our method in general requires more iterations to converge, yet the computation of the iteration is cheaper. This pattern was consistent  across other experiments we conducted. Despite this, the difference in the number of iterations was not significantly enough in order for the better iteration cost to yield a better overall performance. In Experiment 5.8, we noted that as the dimension increases, our method becomes more efficient compared to the method of Qi and Sun as the number of iterations becomes similar. We could not, however, replicate this phenomenon for different values of $\alpha$. This suggests that although we can exploit the diagonal structure of the linear system in order to compute Newton's direction, while Qi and Sun's method resorts to a conjugate gradient method, the subgradient they compute is somewhat more efficient towards solving the problem than the one we obtain in our approach. Specifically, in Experiments 5.6 and 5.7 where the matrix $G$ is far from being positive semidefinite, the better subgradient of Qi and Sun pays off considerably. Nonetheless, our method shows potential superiority for the low noise level regime ($\alpha\approx0$) and high values of $n$ in Experiments 5.5 and 5.8.

\begin{figure}[!ht]
\centering

\begin{subfigure}[b]{0.325\textwidth}
\centering
\includegraphics[scale=0.26]{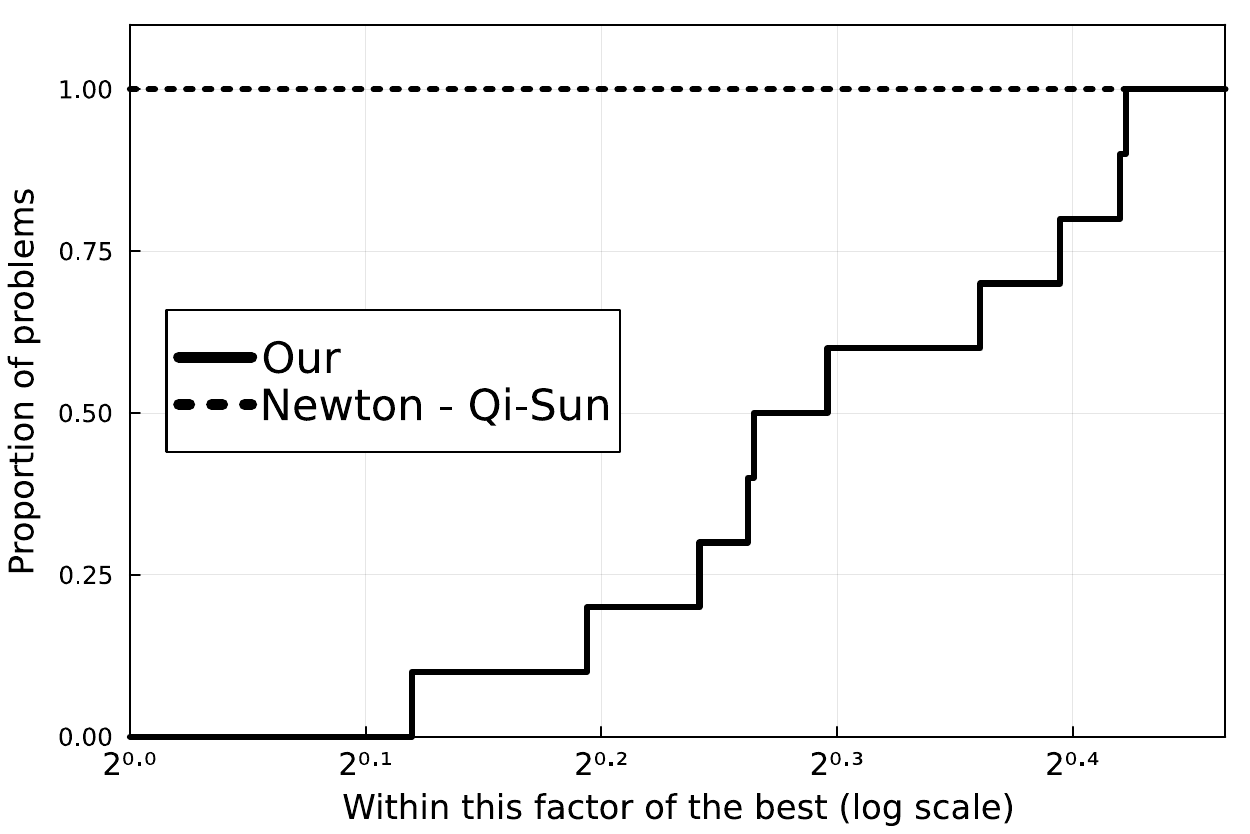}
\caption{$n=5000$}
\label{fig:5k_0001}
\end{subfigure}
\hfill
\begin{subfigure}[b]{0.325\textwidth}
\centering
\includegraphics[scale=0.26]{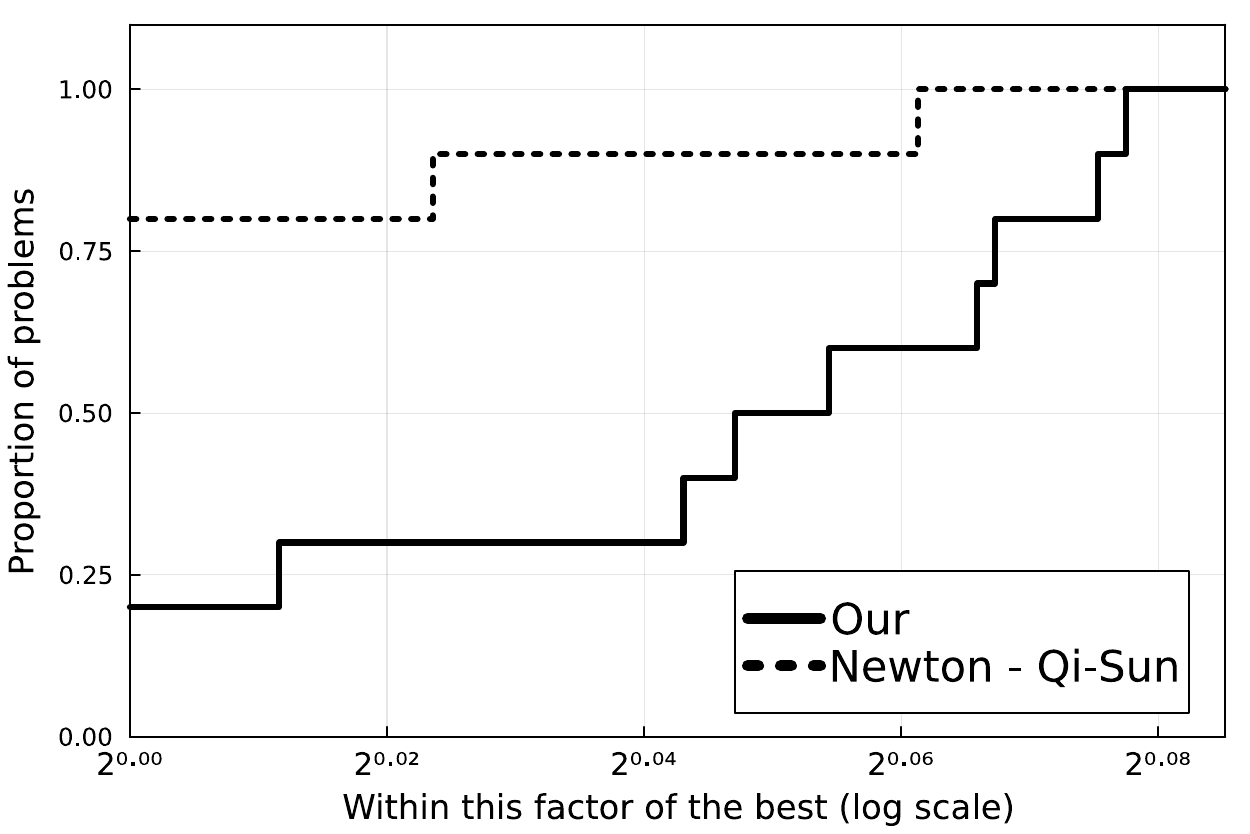}
\caption{$n=8000$}
\label{fig:8k_0001}
\end{subfigure}
\hfill
\begin{subfigure}[b]{0.325\textwidth}
\centering
\includegraphics[scale=0.26]{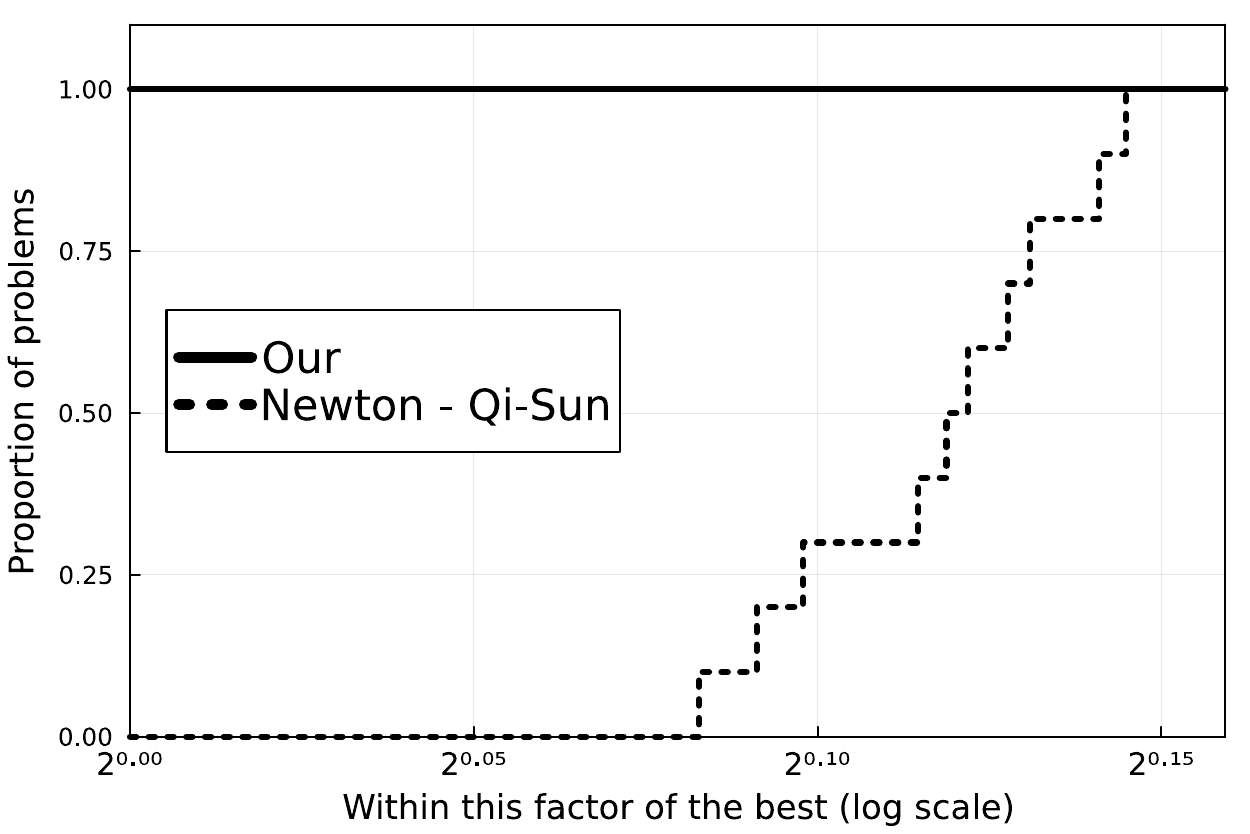}
\caption{$n=10000$}
\label{fig:10k_0001}
\end{subfigure}
\caption{Performance profiles for Experiment 5.8.}
\label{fig:PerfProf}
\end{figure}

\begin{table}[!ht]
\centering
\begin{tabular*}{0.57\textwidth}{lcccccc}
\toprule%
& \multicolumn{2}{@{}c@{}}{$n=5000$} & \multicolumn{2}{@{}c@{}}{$n=8000$} & \multicolumn{2}{@{}c@{}}{$n=10000$} \\\cmidrule{2-4}\cmidrule{5-7}%
Method & time (s) & it & time (s) & it & time (s)& it \\
\midrule
Our & 972,99 &	13 & 3677,27 &	10 & 6880,66 &	10 \\
Qi-Sun & 791,12 &	8	& 3588,47 &	8	& 7463,56 &	9	\\
\hline
\end{tabular*}
\caption{Average total time and number of iterations for $n=5000$, $n=8000$, and $n=10000$.}
\label{fig:Tab_5.8}
\end{table}

%%%%%%%%%%%%%%%%%%%%%%%%%%%%%%%%%%%%%%%%%%%%%
\section{Concluding remarks}
%%%%%%%%%%%%%%%%%%%%%%%%%%%%%%%%%%%%%%%%%%%%%

In this paper, we investigated the global convergence properties of the semi-smooth Newton method when applied to a general projection equation within finite-dimensional spaces. We have further highlighted the intrinsic connection between the solutions of this projection equation and the constrained quadratic conic programming problem. Comparative experiments on semidefinite least squares problems, particularly the nearest correlation matrix problem, benchmarked against existing literature, underscore the efficacy of the proposed method. The methodology introduced herein paves the way for future research into the versatility and performance of the semi-smooth Newton method in addressing a broader conic constrained problems via generalized projection equations.

\end{document}